\documentclass[a4paper,11pt, twoside]{amsart}

\usepackage{amsmath,amsfonts,amsthm,amsopn,color,amssymb,enumitem}
\usepackage{palatino}

\usepackage[colorlinks=true,urlcolor=blue, citecolor=red,linkcolor=blue,linktocpage,pdfpagelabels, bookmarksnumbered,bookmarksopen]{hyperref}
\usepackage[hyperpageref]{backref}

\usepackage[left=2.61cm,right=2.61cm,top=2.72cm,bottom=2.72cm]{geometry}

\newcommand{\R}{\mathbb{R}}

\newcommand{\RN}{{\mathbb{R}^N}}

\renewcommand{\le}{\leslant}
\renewcommand{\ge}{\geslant}
\renewcommand{\a }{\alpha }

\renewcommand{\d }{\delta }

\renewcommand{\l }{\lambda}

\newcommand{\s }{\sigma }
\renewcommand{\t}{\theta}

\renewcommand{\L}{\Lambda}

\newcommand{\I}{\mathcal{I}}

\newcommand{\M}{\mathcal{M}_{\l,\L}}

\newcommand{\N}{\mathbb{N}}

\def\bbm[#1]{\mbox{\boldmath $#1$}}
\newcommand{\beq }{\begin{equation}}
\newcommand{\eeq }{\end{equation}}

\renewcommand{\le}{\leqslant}
\renewcommand{\ge}{\geqslant}
\newcommand{\dis}{\displaystyle}

\providecommand{\pgfsyspdfmark}[3]{}

\renewcommand{\l}{\lambda}

\makeatletter
\@addtoreset{equation}{section}%
\renewcommand{\theequation}{\thesection.\@arabic\c@equation}
\makeatother

\makeatletter
\providecommand\@dotsep{5}
\def\listtodoname{List of Todos}
\def\listoftodos{\@starttoc{tdo}\listtodoname}
\makeatother

\newtheorem{theorem}{Theorem}[section]
\newtheorem{lemma}[theorem]{Lemma}
\newtheorem{remark}[theorem]{Remark}
\newtheorem{proposition}[theorem]{Proposition}
\newtheorem{definition}[theorem]{Definition}

\title
[Oscillating solutions]
{Oscillating solutions for nonlinear equations involving \\the Pucci's extremal operators}

\author[P. d'Avenia]{Pietro d'Avenia}

\address[P. d'Avenia]{\newline\indent
Dipartimento di Meccanica, Matematica e Management
\newline\indent 
Politecnico di Bari
\newline\indent
Via Orabona 4,  70125  Bari, Italy}
\email{\href{mailto:pietro.davenia@poliba.it}{pietro.davenia@poliba.it}}

\author[A. Pomponio]{Alessio Pomponio}

\address[A. Pomponio]{\newline\indent
Dipartimento di Meccanica, Matematica e Management
\newline\indent 
Politecnico di Bari
\newline\indent
Via Orabona 4,  70125  Bari, Italy}
\email{\href{mailto:alessio.pomponio@poliba.it}{alessio.pomponio@poliba.it}}

\thanks{The authors are supported by PRIN 2017JPCAPN {\em Qualitative and quantitative aspects of nonlinear PDEs}}

\subjclass[2010]{35B05, 35J60}
\keywords{Pucci's extremal operators, fully nonlinear operator equations, oscillating solutions.}

\begin{document}

\begin{abstract}
This paper deals with the following nonlinear equations
\[
\mathcal{M}_{\lambda,\Lambda}^\pm(D^2 u)+g(u)=0 \qquad \hbox{ in }\mathbb{R}^N,
\]
where $\mathcal{M}_{\lambda,\Lambda}^\pm$ are the Pucci's extremal operators,
for $N \ge 1$ and under the assumption $g'(0)>0$. We show the existence of oscillating solutions, namely with an unbounded sequence of zeros. Moreover these solutions are periodic, if $N=1$, while they are radial symmetric and decay to zero at infinity with their derivatives, if $N\ge 2$. 
\end{abstract}

\maketitle

\section{Introduction}

Let $0<\l\le \L$ be two given positive real numbers. If $u\in C^2(\RN,\R)$, the Pucci's extremal operators are given by
\begin{equation*}
\M^+(D^2 u)=\l\sum_{e_i<0} e_i+\L\sum_{e_i\ge 0} e_i
\quad\hbox{ and }\quad
\M^-(D^2 u)=\L\sum_{e_i<0} e_i+\l\sum_{e_i\ge 0} e_i
\end{equation*}
where $e_i = e_i(D^2u)$, $i = 1,\ldots,N$, are the eigenvalues of the Hessian matrix $D^2u$. More details and equivalent definitions can be found in the monograph of Caffarelli and Cabr\`e \cite{CC}. Clearly, in the special case $\l  = \L$ the two operators become the same and in particular we have that
\[
\M^+(D^2 u)=\M^-(D^2 u)=\l\Delta u.
\]
The Pucci's extremal operators have natural applications for instance in financial mathematics \cite{ALP} and stochastic control with variable diffusion coefficients \cite{BL1}. From a mathematical point of view, they are very important and well studied because they are prototypes of fully nonlinear uniformly elliptic operators. Indeed even though they retain positive homogeneity and some properties associated to the maximum principle, they are no longer in divergence form, thus deviating in a fundamental manner away from the Laplacian. Moreover, equations with these kinds of operators are not variational.

Many authors, in the last decades, focused their attention on the study of the equation
\begin{equation*}
\M^\pm(D^2 u)+g(u)=0,
\end{equation*}
looking for radial solutions in the whole $\RN$ or in ball with Dirichlet boundary condition. In particular, existence, non-existence and uniqueness of positive radial solutions, when $g(u)=u^p$ or $g(u)=-u+u^p$, for suitable $p>1$, have been obtained in \cite{CL,EFQ,FI,FQ2003,FQ2004,FQT,GLP}.
Observe that, in these cases, $g'(0)\le 0$. 

Up to our knowledge, very few is known, instead, whenever $g'(0)>0$.
This case, at least when $\l=\L$,  is very important since it is related to the study of the propagation of lights beams in a photorefractive crystals when a saturation effect is taken into account (see \cite{MMP} for a more precise description about these phenomena). Under this condition we can find, for example, the well known nonlinear Helmholtz equation \cite{E,EW1,EW2,EW3}. 

Aim of this paper is, therefore, the study of the case $g'(0)>0$. Under suitable assumptions, we will prove the existence of oscillating solutions, which are periodic for $N=1$ and go to zero at infinity with their derivatives for $N>1$, for problem
\begin{equation}\label{M}\tag{$\mathcal{P}^\pm$}
\M^\pm(D^2 u)+g(u)=0, \qquad\hbox{in }\RN. 
\end{equation}
To be more precise, throughout
this paper, in the one-dimensional case, we will assume on the nonlinearity the following hypotheses:
\begin{enumerate}[label=(g\arabic{*}), ref=g\arabic{*}]
\item \label{g1} $g:\R \to \R$ is locally Lipschitz;
\item \label{g2} $g$ is odd;
\item \label{g3} there exists $\a \in  (0, +\infty]$ such that $g$ is positive on $(0, \a)$ and negative on $(\a, +\infty)$;
\end{enumerate}
while, if $N\ge 2$, we will require in addition that
\begin{enumerate}[label=(g\arabic{*}), ref=g\arabic{*}]
\setcounter{enumi}{3}
\item \label{g4} $g$ is differentiable at $0$ and $g'(0) > 0$.
\end{enumerate}
Moreover, in the following we will denote by  $\displaystyle G(t)=\int_0^t g(s)\ ds$.

In order to state our results, we first need the following
\begin{definition}
A  solution $u$ of \eqref{M} is called {\em oscillating} if it has an unbounded sequence of zeros. It is called {\em localized} when it converges to zero at infinity together with its partial derivatives up to order~2.
\end{definition}

Our main result is
\begin{theorem}\label{main}
Assume \eqref{g1}-\eqref{g3} and, if $N\ge 2$, also \eqref{g4}. Then
there exist infinitely many oscillating solutions of class $C^2(\RN)$ for \eqref{M}. Moreover these solutions are periodic if $N=1$, while they are localized if $N\ge 2$.
\end{theorem}

This theorem will be an immediate consequence of several distinct results (see Theorems \ref{n=1} and \ref{n=1-}, for the one-dimensional case, and Theorems \ref{n>1+} and \ref{n>1-}, for the multidimensional case).

In this paper we will focus our attention when $\lambda < \Lambda$. As already observed, if $\lambda = \Lambda$, we get  the Laplacian and this kind of problem with the same assumptions on $g$ has been studied in \cite{MMP} (see also \cite{GZ}, for a related problem, and \cite{P}, for the prescribed mean curvature equations both in the Euclidean case and in the Lorentz-Minkowski one).  We will see that, comparing our results with those of \cite{MMP,P},  the situation for an equation involving the Pucci's extremal operators is richer than in presence of the classic Laplacian or the mean curvature operator. Moreover, some of our arguments here are quite different from those developed in \cite{GZ,MMP}, both the in one-dimensional case and in the multi-dimensional one. 
\\
For $\l=\L$, namely when we have the Laplacian operator, if $N=1$, up to the trivial cases, the solutions can be either diverging, positively or negatively, or oscillating. In presence of the Pucci's extremal operators, instead, since
\[
\M^+(u'')=
\begin{cases}
\L u'' & \hbox{ if } u'' \ge 0,\\
\l u'' & \hbox{ if } u'' < 0,
\end{cases}
\quad \hbox{ and }\quad
\M^-(u'')=
\begin{cases}
\l u'' & \hbox{ if } u'' \ge 0,\\
\L u'' & \hbox{ if } u'' < 0,
\end{cases}
\]
it is easy to see that when a solution changes sign, it changes also the concavity and so it solves ordinary differential equations with different coefficients: this produces also solutions with horizontal asymptote.
\\
In the multi-dimensional case, in \cite{GZ,MMP} the authors look for radial solutions and, assuming for simplicity that $\l=\L=1$, for any solution $u$, they introduce the {\em energy} function 
\[
E(r)=\frac{1}{2}(u'(r))^2+G(u(r)),
\] 
which is decreasing, whenever $u\neq 0$. This information is a fundamental tool in all their proofs.  
Instead, in the case $\l\neq \L$, if one looks for radial solutions for \eqref{M} for example in  the case $\M^+$ (the other one is similar), then for  a function $u$ of class $C^2$, we have 
\[
\M^+(D^2 u)(r)= \theta u''(r)+\frac{N-1}{r}\Theta u'(r),
\]
where
\[
\begin{array}{lllll}
\theta=\Lambda& \hbox{when }u''\ge 0 & \hbox{and}& \theta=\lambda & \hbox{when }u''<0; 
\\
\Theta=\Lambda& \hbox{when }u'\ge 0 & \hbox{and}& \Theta=\lambda & \hbox{when }u'<0.
\end{array}
\]
Hence, if $u$ is a solution of \eqref{M}, then one may define the  {\em energy} function 
\[
E_\theta(r)=\frac{\theta}{2}(u'(r))^2+G(u(r)).
\] 
However, as already noticed in \cite{FQT},  we cannot deduce that $E_\theta$ is a decreasing function over the whole range where $u$ is defined and different from zero. In fact $E_\theta$ can be discontinuous at points where $u$ changes concavity. 
Thus $E_\theta$ is only a piecewise $C^1$ function and decreases over each subinterval where $u$ has the same concavity. Therefore $E_\theta$ should be replaced by $E_\l$ and $E_\L$ but only partial results on the monotonicity of $E_\l$ and $E_\L$ are known and this requires a more precise and delicate analysis.

The paper is organized as follows: in Section \ref{se1}, we deal the one-dimensional case, while in Section \ref{se2} we treat the multi-dimensional case.

Finally, $c, c_i$ are fixed independent constants which may vary from line to line.

\section{The case $N=1$}\label{se1}

In this section we treat the one-dimensional case. We start looking for even solutions of \eqref{M} in the case $\M^+$, considering the following Cauchy problem
\begin{equation}
\label{CP1+}
\begin{cases}
\M^+(u'')+g(u)=0,\\
u(0)=\xi,\\
u'(0)=0,
\end{cases}
\end{equation}
where we recall that
\[
\M^+(u'')=
\begin{cases}
\L u'' & \hbox{ if } u'' \ge 0,\\
\l u'' & \hbox{ if } u'' < 0.
\end{cases}
\]

Our main result in this case is the following

\begin{theorem}\label{n=1}
Assume \eqref{g1}-\eqref{g3}. For any $\xi \in \R$ there exists a solution $u_\xi\in C^2([0,R_\xi))$ of the Cauchy problem \eqref{CP1+}, where $R_\xi\in(0,+\infty]$ is such that $[0,R_\xi)$ is the maximal interval where the function $u_\xi$ is defined. Moreover, we have
\begin{enumerate}[label=(\roman{*}), ref=\roman{*}]
	\item \label{i1+} if $|\xi|=\a \in \R$ or $\xi=0$, then $u_\xi\equiv \xi$;
	\item \label{ii1+} if $\a \in \R$ and  $|\xi| > \a $, then $|u_\xi|$ strictly increases to $+\infty$ on $[0, R_\xi )$;
	\item \label{iii1+} if $0 < \xi < \a$, we distinguish three cases:
\begin{enumerate}
		\item if $\Lambda G(\xi) < \lambda G(\alpha)$, then  $R_\xi=+\infty$ and $u_\xi$ is oscillating and periodic with $\|u_\xi\|_{L^\infty(\R_+)}=
-u_\xi(r_2)<\alpha$, where $r_2$ is the first minimizer of $u$;
		\item if $\Lambda G(\xi) = \lambda G(\alpha)$, then $R_\xi=+\infty$ and  $-\alpha <u_\xi<\alpha$ in $\R_+$; moreover $u_\xi$ is decreasing in $\R_+$  and $\lim_{r\to + \infty} u_\xi(r)=-\alpha$;
		\item if $\Lambda G(\xi) > \lambda G(\alpha)$, then $u_\xi$  decreases to $-\infty$ on $[0, R_\xi )$;
	\end{enumerate}
\item \label{iv1+}  if $-\a < \xi < 0$, then $R_\xi=+\infty$ and $u_\xi$ is oscillating and periodic with $\|u_\xi\|_{L^\infty(\R_+)}=-\xi$.
\end{enumerate}
\end{theorem}

\begin{proof}
It is standard to see that there exists a unique local solution $u_\xi$ of the Cauchy problem \eqref{CP1+}. Now let $R_\xi>0$ be such that $[0,R_\xi)$ is the maximal interval where the function $u_\xi$ is defined. We have $u_\xi\in C^2([0,R_\xi))$. In the following we simply write $u,R$ instead of $u_\xi,R_\xi$, for brevity.
\medskip
\\
If ${|\xi|}=\alpha\in \R$ or $\xi=0$ then, by the assumptions on $g$ we have the constant solution $u(r)=\xi$ and so (\ref{i1+}) follows immediately. 
\medskip
\\
Let us prove \eqref{ii1+},  when $\xi>\a$. The case $\xi<-\alpha$ is analogous.\\
If $\xi >\alpha$, then $u''(0)>0$ and so $u'(r)>0$ for $r$ small. Let now define
\[
\bar r:=\sup\left\{r\in[0,R): u'>0 \hbox{ in } (0,r) \right\}.
\]
If $\bar r=R$, then $u$ is strictly convex in $\R_+$ and  the proof is concluded. Assume instead, by contradiction, that  $\bar r<R$.
In $[0,\bar r]$ we have that
$u''>0$. Then $u$ satisfies $-\Lambda u''=g(u)$ and, since
\[
u'(\bar r)=\int_0^{\bar r} u''(s) ds >0,
\]
we reach a contradiction due to the definition of $\bar r$.
\medskip
\\
Let us prove (\ref{iii1+}).\\
Since $g(u(0))=g(\xi)>0$ then $u''(0) < 0$ and $u'(r)<0$, for $r$ small. 
But the assumptions on $g$ imply that $u''<0$ whenever $0 < u(r) <  \alpha$ and so $u$ remains decreasing as long it remains positive. \\
The concavity of $u$, whenever $u$ is positive, implies also that $u$ cannot have a horizontal asymptote.
Hence there exists $r_1 \in (0,R)$ such that $u(r_1)=0$.
Since $u$ is the solution of
\begin{equation}\label{cacca}
\begin{cases}
-\lambda u''=g(u), \quad\hbox{ in }[0,r_1],\\
u(0)=\xi,\\
u'(0)=0,\\
u(r_1)=0,
\end{cases}
\end{equation}
multiplying the equation in \eqref{cacca} by $u'$ and integrating in $[0,r_1]$, we have that
\[
\frac{1}{2} (u'(r_1))^2
= \int_0^{r_1} u'(s)u''(s)\ ds 
= -\frac{1}{\lambda}  \int_0^{r_1} g(u(s)) u'(s)\ ds
= \frac{1}{\lambda}  G(\xi),
\]
which implies that 
\begin{equation}\label{u'r1-1+}
u'(r_1)=-\sqrt{2 G(\xi)/\lambda}.
\end{equation}
Observe that, since $u'(r_1)<0$, then $u$ is negative in a right neighbourhood of $r_1$. Using again the assumptions on $g$, it is also convex  whenever  $-\alpha< u(r) < 0$, and, therein, $u$ solves
\begin{equation}\label{dopor1-1+}
\begin{cases}
-\Lambda u''=g(u),\\
u(r_1)=0,\\
u'(r_1)=-\sqrt{2 G(\xi)/\lambda}.
\end{cases}
\end{equation}
Moreover, for any $r$ in such an interval, multiplying the equation in \eqref{dopor1-1+} by $u'$, integrating in $[r_1,r]$ and by \eqref{u'r1-1+}, we infer  that
\begin{equation}
\label{eq_energy+}
\frac{1}{2} (u'(r))^2 - \frac{G(\xi)}{\lambda} = \int_{r_1}^{r} u'(s)u''(s) \ ds = -\frac{1}{\Lambda}  \int_{r_1}^{ r} g(u(s)) u'(s)\ ds = -\frac{ G(u(r))}{\Lambda} .
\end{equation}
We now distinguish three cases.\\
{\bf Case 1:} $\Lambda G(\xi) < \lambda G(\alpha)$.\\
Let us prove that $u$ does not touch $-\alpha$ remaining decreasing (and so convex).\\
Assume, by contradiction, that there exists $\bar{r}:=\min\{r\in \R_+ : u(r)=-\alpha \}$. By \eqref{eq_energy+}, we deduce that
\begin{equation}
\label{barr+}
0 \le \frac{1}{2} (u'(\bar{r}))^2
= \frac{G(\xi)}{\lambda}  -\frac{ G(\alpha)}{\Lambda}
\end{equation}
and we reach a contradiction with $\Lambda G(\xi) < \lambda G(\alpha)$. 
\\
Let us prove that there exists $r_2>r_1$ which is an isolated minimizer of $u$.\\
Assume, by contradiction, that $u$ remains decreasing for every $r>r_1$. Then there exists
\begin{equation*}
\lim_{r\to +\infty} u(r) = \inf\{u(r) : r\in \R_+\}=:u_\infty \in [-\alpha,0).
\end{equation*}
In this case, by the assumption on $g$, we would infer also that $u$ is convex in $[r_1,+\infty)$.\\
If $u_\infty\neq -\alpha$, then  
\[
\lim_{r\to +\infty} u''(r) = - \frac{g(u_\infty)}{\Lambda}>0.
\]
This implies that 
\[
\lim_{r\to +\infty} u'(r)=+\infty,
\]
which is in contradiction with the fact that  $u$ is decreasing.\\
If, instead, $u_\infty=-\alpha$, by \eqref{eq_energy+}, we have
\[
\lim_{r\to + \infty} (u'(r))^2
= 2  \left(\frac{G(\xi)}{\lambda}  -\frac{ G(\alpha)}{\Lambda}\right) < 0,
\] 
reaching again a contradiction.\\
Thus, let  $r_2:=\min\{r\in \R_+: u'(r)=0\}$. Clearly we have that $u(r_2)\in (-\alpha,0)$ and $u''(r_2)>0$.
\\
It is standard to prove that  $u$ is periodic with period $2r_2$.
\\
Finally we observe that, again by \eqref{eq_energy+},
\[
\frac{G(\xi)}{\lambda}
= \frac{ G(u(r_2))}{\Lambda},
\]
so that
\[
G(u(r_2))
= \frac{\Lambda}{\lambda} G(\xi)
> G(\xi)
\]
which implies
\[
\|u\|_{L^\infty(\R_+)} =\max \{\xi, -u(r_2)\}= -u(r_2).
\]
{\bf Case 2:} $\Lambda G(\xi) = \lambda G(\alpha)$.\\ 
We prove that there cannot exist $\tilde{r}:=\min\{r\in \R_+:u(r)\in (-\alpha,0),   u'(r)=0\}$.\\
Indeed, if it existed, 
by \eqref{eq_energy+}, we would get
\[
0
= \frac{1}{2} (u'(\tilde{r}))^2 = \frac{G(\xi)}{\lambda} -\frac{ G(u(\tilde{r}))}{\Lambda}
> \frac{G(\xi)}{\lambda} -\frac{ G(\alpha)}{\Lambda}
=0,
\]
reaching a contradiction.
\\
Therefore two different cases can occur: either there exists $\bar{r}:=\min\{r\in \R_+ : u(r)=-\alpha \}$, or $u>-\alpha$ and decreasing in $\R_+$.\\
In the first case, arguing as in \eqref{barr+}, we have $u'(\bar{r})=0$ and so, in a left neighborhood of $\bar r$,  $u$ and the constant function $v\equiv -\alpha$ would be two different solutions of the same Cauchy problem
\[
\begin{cases}
-\L u''=g(u),
\\
u(\bar r)=-\alpha,
\\
u'(\bar r)=0,
\end{cases}
\]  
reaching a contradiction with the uniqueness of the solution.\\
Therefore, the second case holds, $u$ must have a horizontal asymptote, and, arguing as in {\bf Case 1}, we deduce that $\lim_{r\to + \infty}u(r)=-\alpha$.
\\
{\bf Case 3:} $\Lambda G(\xi) > \lambda G(\alpha)$.\\
Repeating the arguments of {\bf Case 2},  there cannot exist $\tilde{r}:=\min\{r\in \R_+:u(r)\in (-\alpha,0),   u'(r)=0\}$.\\
Moreover $u$ cannot remain in $(-\alpha,0)$, being also decreasing for $r>r_1$.\\
Indeed, otherwise, $u$ should have a horizontal asymptote and, arguing as in {\bf Case 1} we would deduce that $\lim_{r\to + \infty}u(r)=-\alpha$. But,
by  \eqref{eq_energy+}, there would exist
\[
\lim_{r\to + \infty} (u'(r))^2 = 2  \left(\frac{G(\xi)}{\lambda}  -\frac{ G(\alpha)}{\Lambda}\right) >0.
\]
Thus, for a suitable constant $c>0$ and for $r$ large enough, $u'(r)<-c$, that implies that $u$ diverges to $-\infty$, reaching a contradiction. 
\\
Therefore there exists $\bar{r}:=\min\{r\in \R_+ : u(r)=-\alpha \}$. Arguing as in \eqref{barr+} we have that 
$u'(\bar{r})<0$ and so $u(r)<-\alpha$ in a right neighbourhood of $\bar{r}$. In such a neighbourhood we have that $u''(r)<0$ and so $u$ solves $-\lambda u''=g(u)$. Now we can proceed arguing as in the proof of \eqref{ii1+} and we can conclude.
\medskip \\
Let us prove (\ref{iv1+}).\\
Arguing as in the first part of the proof of (\ref{iii1+}), we can prove the existence of a zero $r_1$ such that
\begin{equation}\label{u'r1-1+n}
u'(r_1)=\sqrt{2 G(\xi)/\Lambda}
\end{equation}
and that in the right neighbourhood of $r_1$ where $0< u(r) < \alpha$, $u$ solves
\begin{equation}\label{dopor1-1+n}
\begin{cases}
-\lambda u''=g(u),\\
u(r_1)=0,\\
u'(r_1)=\sqrt{2 G(\xi)/\Lambda}.
\end{cases}
\end{equation}
Moreover, for any $r$ in such an interval, multiplying the equation in \eqref{dopor1-1+n} by $u'$, integrating in $[r_1,r]$ and by \eqref{u'r1-1+n}, we infer  that
\begin{equation}
\label{eq_energy+n}
\frac{1}{2} (u'(r))^2 - \frac{G(\xi)}{\Lambda} = \int_{r_1}^{r} u'(s)u''(s) \ ds = -\frac{1}{\lambda}  \int_{r_1}^{ r} g(u(s)) u'(s)\ ds = -\frac{ G(u(r))}{\lambda} .
\end{equation}
Assume by contradiction that there exists 
$\bar r:=\min\{r>r_1: u(r)=-\xi\}$.
%
%
Then, by \eqref{eq_energy+n} we have $(u'(\bar{r}))^2<0$  which is a contradiction.\\
Moreover, arguing as in {\bf Case 1}, $u$ cannot remain increasing in $(r_1,+\infty)$ and so there exists $r_2>r_1$ such that $u(r_2)\in (0,-\xi)$, $u'(r_2)=0$ and $u''(r_2)<0$ ($r_2$ is an isolated maximizer). 
Also in this case $u$ is periodic with period $2r_2$ and
$\|u\|_{L^\infty(\R_+)}
= -\xi$.
\end{proof}

We conclude this section treating briefly the case $\M^-$. 

Let us consider
\begin{equation}
\label{CP1-}
\begin{cases}
\M^-(u'')+g(u)=0,\\
u(0)=\xi,\\
u'(0)=0,
\end{cases}
\end{equation}
where we recall that
\[
\M^-(u'')=
\begin{cases}
\l u'' & \hbox{ if } u'' \ge 0,\\
\L u'' & \hbox{ if } u'' < 0.
\end{cases}
\]
Observing that, by \eqref{g2}, if $u$ is a solution of 
\begin{equation*}
\begin{cases}
\M^-(u'')+g(u)=0,
\\
u(0)=\xi,
\\
u'(0)=0,
\end{cases}
\end{equation*}
then, clearly, $-u$ is a solution of 
\begin{equation*}
\begin{cases}
\M^+(u'')+g(u)=0,
\\
u(0)=-\xi,
\\
u'(0)=0,
\end{cases}
\end{equation*}
by Theorem \ref{n=1}, we obtain  immediately 
\begin{theorem}\label{n=1-}
Assume \eqref{g1}-\eqref{g3}. For any $\xi \in \R$ there exists a solution $u_\xi\in C^2([0,R_\xi))$ of the Cauchy problem \eqref{CP1-}, where $R_\xi\in(0,+\infty]$ is such that $[0,R_\xi)$ is the maximal interval where the function $u_\xi$ is defined. Moreover, we have
\begin{enumerate}[label=(\roman{*}), ref=\roman{*}]
	\item \label{i1-} if $|\xi|=\a \in \R$ or $\xi=0$, then $u_\xi\equiv \xi$;
	\item \label{ii1-} if $\a \in \R$ and  $|\xi| > \a $, then $|u_\xi|$ strictly increases to $+\infty$ on $[0, R_\xi )$;
\item \label{iii1-}  if $0 < \xi < \a$, then $R_\xi=+\infty$ and $u_\xi$ is oscillating and periodic with $\|u_\xi\|_{L^\infty(\R_+)}=\xi$.
	\item \label{iv1-} if $-\a < \xi < 0$, we distinguish three cases:
\begin{enumerate}
		\item if $\Lambda G(\xi) < \lambda G(\alpha)$, then  $R_\xi=+\infty$ and $u_\xi$ is oscillating and periodic with $\|u_\xi\|_{L^\infty(\R_+)}=u_\xi(r_2)<\alpha$, where $r_2$ is the first maximizer of $u_\xi$;
		\item if $\Lambda G(\xi) = \lambda G(\alpha)$, then $R_\xi=+\infty$ and  $-\alpha <u_\xi<\alpha$ in $\R_+$; moreover $u_\xi$ is increasing in $\R_+$ and $\lim_{r\to + \infty} u_\xi(r)=\alpha$;
		\item if $\Lambda G(\xi) > \lambda G(\alpha)$, then  $u_\xi$  increases to $+\infty$ on $[0, R_\xi )$.
	\end{enumerate}
\end{enumerate}
\end{theorem}

\section{The case $N>1$}\label{se2}

In this section we deal with the multidimensional case. As observed, for example, in \cite{FQT}, we recall, that if $N>1$  and  if $u$ is a $C^2$ radial function of $\RN$, with the abuse of notation $u(x)=u(r)$, with $r=|x|$, we have 
\[
\M^+(D^2 u)(r)= \theta u''(r)+\frac{N-1}{r}\Theta u'(r),
\]
where
\[
\begin{array}{lllll}
\theta=\Lambda& \hbox{when }u''\ge 0 & \hbox{and} & \theta=\lambda & \hbox{when }u''<0; 
\\
\Theta=\Lambda & \hbox{when }u'\ge 0 & \hbox{and} & \Theta=\lambda & \hbox{when }u'<0.
\end{array}
\]
Therefore, since we are interested in radial solutions of \eqref{M}, in the case $\M^+$,  we will study the following  initial value problem of the ordinary differential equation
\begin{equation}\label{eqr}
\begin{cases}
\dis \theta u''+\frac{N-1}{r}\Theta u'+g(u)=0,
\\
u(0)=\xi,
\\
u'(0)=0,
\end{cases}
\end{equation}
for a certain $\xi\in \R$.

Our main result for $\M^+$, is

\begin{theorem}\label{n>1+}
Assume \eqref{g1}-\eqref{g4}. For any $\xi \in \R$ there exists a solution $u_\xi\in C^2([0,R_\xi))$ of the Cauchy problem (\ref{eqr}), where $R_\xi\in (0,+\infty]$ is such that $[0,R_\xi)$ is the maximal interval where the function $u_\xi$ is defined. Moreover, we have
\begin{enumerate}[label=(\roman{*}), ref=\roman{*}]
\item \label{n>1+i} if $|\xi|=\a \in \R$ or $\xi=0$, then $u_\xi\equiv \xi$;
\item \label{n>1+ii} if $\a \in \R$ and  $|\xi| > \a $, then $|u_\xi|$ strictly increases to $+\infty$ on $[0, R_\xi )$;
\item \label{n>1+iii} if $0 < \xi < \a$ and such that
\begin{equation}\label{xi}
\Lambda G(\xi)\le\lambda G(\alpha),
\end{equation}
then $R_\xi=+\infty$ and $u_\xi$ is oscillating and localized with $\|u_\xi\|_{L^\infty(\R_+)}<\alpha$;
\item \label{n>1+iv}if $-\alpha < \xi < 0$, then $R_\xi=+\infty$ and $u_\xi$ is oscillating and localized with $\|u_\xi\|_{L^\infty(\R_+)}=-\xi<\alpha$.
\end{enumerate}
\end{theorem}

For what concerns $\M^-$, instead, we have that
\[
\M^-(D^2 u)(r)= \psi u''(r)+\frac{N-1}{r}\Psi u'(r),
\]
where
\[
\begin{array}{lllll}
\psi=\lambda & \hbox{when }u''\ge 0 & \hbox{and} & \psi=\Lambda & \hbox{when }u''<0; 
\\
\Psi=\lambda & \hbox{when }u'\ge 0 & \hbox{and} & \Psi=\Lambda & \hbox{when }u'<0.
\end{array}
\]
Thus we consider the following  initial value problem of the ordinary differential equation
\begin{equation}\label{eqr-}
\begin{cases}
\dis \psi u''+\frac{N-1}{r}\Psi u'+g(u)=0,
\\
u(0)=\xi,
\\
u'(0)=0,
\end{cases}
\end{equation}
for a certain $\xi\in \R$. In this case, observing that, if $u$ is a solution of \eqref{eqr-}, then $-u$ is a solution of \begin{equation*}
\begin{cases}
\dis \theta u''+\frac{N-1}{r}\Theta u'+g(u)=0,
\\
u(0)=-\xi,
\\
u'(0)=0,
\end{cases}
\end{equation*}
as an immediate consequence of Theorem \ref{n>1+}, we have
\begin{theorem}\label{n>1-}
Assume \eqref{g1}-\eqref{g4}. For any $\xi \in \R$ there exists a solution $u_\xi\in C^2([0,R_\xi))$ of the Cauchy problem (\ref{eqr-}), where $R_\xi>0$ is such that $[0,R_\xi)$ is the maximal interval where the function $u_\xi$ is defined. Moreover, we have
\begin{enumerate}[label=(\roman{*}), ref=\roman{*}]
\item \label{n>1-i} if $|\xi|=\a \in \R$ or $\xi=0$, then $u_\xi\equiv \xi$;
\item \label{n>1-ii} if $\a \in \R$ and  $|\xi| > \a $, then $|u_\xi|$ strictly increases to $+\infty$ on $[0, R_\xi )$;
\item \label{n>1-iii} if $0 < \xi < \a$
then $R_\xi=+\infty$ and $u_\xi$ is oscillating and localized with $\|u_\xi\|_{L^\infty(\R_+)}=\xi<\alpha$;
\item \label{n>1-iv}if $-\alpha < \xi < 0$,  and such that
\begin{equation*}
\Lambda G(\xi)\le\lambda G(\alpha),
\end{equation*}
then $R_\xi=+\infty$ and $u_\xi$ is oscillating and localized with $\|u_\xi\|_{L^\infty(\R_+)}<\alpha$.
\end{enumerate}
\end{theorem}

Therefore, in what follows, we will prove only Theorem \ref{n>1+}. 

As already observed in the Introduction, we have to deal with the following two {\em energy} functions 
\begin{equation*}
E_\lambda (r)=\frac{\lambda}{2}(u'(r))^2+G(u(r)) \
\hbox{ and } \ 
E_\Lambda (r)=\frac{\Lambda}{2}(u'(r))^2+G(u(r)) .
\end{equation*}

Some partial informations on the monotonicity properties of $E_\l$ and $E_\L$ are known. Indeed we have
\begin{lemma}[\cite{FQT}, Lemma 2.1]\label{le:Lemma21+}
	Let $u$ be a solution of \eqref{eqr} in the case $\M^+$ and $I\subset \R$ an interval. Then
	\begin{enumerate}[label=(\roman{*}), ref=\roman{*}]
		\item if $u'>0$ in $I$, then both $E_\lambda(r)$ and $E_\Lambda(r)$ are decreasing in $I$;
		\item if $u' < 0$ in $I$, then $E_\lambda(r)$ decreases when $g(u) > 0$ and $E_\Lambda(r)$ decreases when $g(u) < 0$.
	\end{enumerate}
\end{lemma}

Arguing as in \cite{FQ2004,NN}, there exists a local solution $u_\xi$ of the Cauchy problem \eqref{eqr}. Now let $R_\xi\in (0,+\infty]$ be such that $[0,R_\xi)$ is the maximal interval where the function $u_\xi$ is defined. We have $u_\xi\in C^2([0,R_\xi))$. In the following we simply write $u,R$ instead of $u_\xi,R_\xi$, for brevity.

The proof of case (\ref{n>1+i}) follows immediately by the assumptions on $g$. For what concerns the other three points, since the proof is quite long and involved, we divide it into three different sections.

\subsection{Proof of (\ref{n>1+ii}) of Theorem \ref{n>1+}}

\
\\
Let us prove (\ref{n>1+ii}) in the case $\xi>\a $. The case $\xi<-\a $ is analogous.
\\
We first prove that $u$ is strictly increasing in $[0,R)$. 
\\
We can find a sequence $\{s_n\}_n\subset \R_+$, such that $s_n \searrow 0$ and $u'(s_n)$ and $u''(s_n)$ have fixed sign and so that, $\theta$ and $\Theta$, evaluated on $s_n$, do not depend on $n$. Then, since
\begin{equation}
\label{behav0}
\begin{split}
-g(\xi)
&=-g(u(0))
=\lim_{n}\left( \theta u''(s_n)+\frac{N-1}{s_n} \Theta u'(s_n)\right)
\\
&=\lim_{n}\left( \theta u''(s_n)+(N-1)\Theta\frac{u'(s_n)-u'(0)}{s_n}\right)
=\big(\theta+(N-1)\Theta\big)u''(0),
\end{split}
\end{equation}
we have that $u''(0)>0$. Therefore $u$ is convex and strictly increasing in a right neighbourhood of $0$ and, by \eqref{behav0}, we have $-g(\xi)=N\L u''(0)$.\\
Suppose, by contradiction that $u$ is not always increasing, then  there exists $\bar r\in [0,R)$, such that $u(\bar{r})>\xi$, $u'(\bar r)=0$, $u''(\bar r)\le 0$. Then 
\[
\underbrace{\theta u''(\bar r)}_{\le 0}+\underbrace{\frac{N-1}{\bar r}\Lambda u'(\bar r)}_{=0}+\underbrace{g(u(\bar r))}_{<0}=0,
\]
which is a contradiction and hence we have that $u$ is increasing.\\
We want to prove that $u$ diverges positively. 
\\
Assume by contradiction that there exists
\begin{equation}\label{asintoto}
\lim_{r\to +\infty} u(r)=u_\infty\in\R,
\end{equation}
where, clearly, $u_\infty>\xi>\alpha$.
Since $u$ is increasing in $\R_+^*$, then, by Lemma \ref{le:Lemma21+},  $E_\lambda$ is decreasing in $\R_+^*$. Moreover, by \eqref{asintoto}, we have that $G(u(r)) \to G(u_\infty)$ as $r \to +\infty$. Thus there exists
\[
\lim_{r \to +\infty} (u'(r))^2 \in \mathbb{R}
\]
and so, since $u'(r)>0$, there exists
\begin{equation}
\label{asinup}
\lim_{r \to +\infty} u'(r) \in \mathbb{R}.
\end{equation}
Hence, combining \eqref{asintoto} and \eqref{asinup} we immediately obtain that
\begin{equation*}
\lim_{r\to +\infty} u'(r)=0.
\end{equation*}
Observe now that, by \eqref{asintoto}
there exists a sequence $r_n\to+\infty$ such that $u''(r_n)<0$.
Then
\[
0>\lambda u''(r_n)
=- \Lambda \frac{N-1}{r_n} u'(r_n) -g(u(r_n))
\to -g(u_\infty)>0,
\]
which gives a contradiction.

\subsection{Proof of  (\ref{n>1+iii}) of Theorem \ref{n>1+}}

\
\\
Arguing as in \eqref{behav0}, we infer that $N\l u''(0)=-g(\xi)<0$
and $u$ is decreasing and concave in a right neighbourhood of $0$.\\
We divide the proof into some intermediate lemmas.

\begin{lemma}\label{le:step1}
$u'<0$ until $u$ attains a first zero in a certain $r_1>0$.
\end{lemma}
\begin{proof}
If $u$ remains decreasing and  concave, then it necessarily has to decrease to a first zero. 
\\
Assume, by contradiction that $u>0$ in $[0,R)$. This implies that $u'<0$ therein. Indeed, there cannot exist $\bar r\in (0,R)$ such that $u(\bar r)\in (0,\a) $, $u'(\bar r)=0$ and $u''(\bar r)\ge 0$ because, otherwise,
\[
\underbrace{\Lambda u''(\bar r)}_{\ge 0}+\underbrace{\frac{N-1}{\bar r}\lambda u'(\bar r)}_{=0}+\underbrace{g(u(\bar r))}_{>0}=0,
\]
which is impossible.
So, since $u(r)\in (0,\a)$ for any $r\in [0,R)$, then $R=+\infty$ and  there exists 
\[
\lim_{r\to +\infty}u(r)=u_\infty\in[0,\alpha).
\]
Therefore, $u$ satisfies
\begin{equation*}
u''=-\frac{N-1}{\theta r}\lambda u'-\frac{g(u)}{\t}
\le -\frac{N-1}{ r} u'-\frac{g(u)}{\Lambda}.
\end{equation*}
If we set $v(r) = r^\frac{N-1}{ 2} u(r)$,  we get the following
\begin{equation}\label{sigma+}
v''+\s _+(r)v\le  0 \qquad\hbox{ where }\
\s_+(r):= 
\frac{ g(u(r))}{\Lambda u(r)}-\frac{(N -1)(N-3)}{4 r^2}.
\end{equation}
By \eqref{g4}, we infer that there exists $c_0>0$ such that $\s_+(r)\ge c_0$, definitively, and so $v''$ is definitively negative.
Being $v'$ definitively decreasing, there exists $L=\lim_{r\to +\infty}v'(r)<+\infty$. Observe that $L\ge 0$, because, otherwise, $\lim_{r\to +\infty}v(r)=-\infty$ contradicting the positivity of $v$.
Hence $v$ is definitively increasing and so there exists $c_1>0$ such that $v(r)\ge c_1$, definitively. This, together with \eqref{sigma+}, implies that $L=-\infty$ which clearly is not possible.
\end{proof}

\begin{lemma}\label{le:step2}
 $u'<0$ until $u$  attains a critical point (which is also a local minimizer) at some $r_2 > r_1$ with $-\alpha < u(r_2 ) < 0$.
\end{lemma}

\begin{proof}
	Concerning the behavior of $u$ on $[r_1 ,+\infty)$, there are three possibilities:
	\begin{enumerate}[label=(\alph{*}), ref=\alph{*}]
		\item \label{Nplusa} $u'<0$ until $u$ attains $-\alpha$  at some $\bar r>r_1$;
		\item \label{Nplusb} $u'<0$ on $[r_1 , +\infty)$ and $u$ decreases to some value $u_\infty \in  [-\alpha, 0)$;
		\item \label{Nplusc} $u'<0$ until $u$ attains a critical point (which is a minimizer) at some $r_2 > r_1$ with $-\alpha < u(r_2 ) < 0$.
	\end{enumerate}
	Let us show that the cases (\ref{Nplusa}) and (\ref{Nplusb}) do not occur.
	\\
	Suppose that (\ref{Nplusa}) holds. So, there exists $\bar r>r_1$ such that $u(\bar r)=-\alpha$ and  $u'< 0$ in $[r_1,\bar r]$.
	\\
	By Lemma \ref{le:Lemma21+}, $E_\lambda$ is decreasing in $[0,r_1]$ and so $E_\lambda (0)>E_\lambda(r_1)$, which implies that
	\begin{equation}\label{+1}
	G(\xi)>\frac \lambda 2 (u'(r_1))^2.
	\end{equation} 
	Moreover, again by Lemma \ref{le:Lemma21+}, $E_\Lambda$ is decreasing in $[r_1,\bar r]$ and so we have that $E_\Lambda(r_1)>E_\Lambda(\bar r)$, namely
	\begin{equation}\label{+2}
	\frac \Lambda 2(u'(r_1))^2>
	\frac \Lambda 2 (u'(\bar r))^2 +G(\alpha).
	\end{equation} 
	Therefore, by \eqref{+1} and \eqref{+2}, we have
	\[
	G(\alpha)
	\le \frac \Lambda 2 (u'(\bar r))^2 +G(\alpha)
	<\frac \Lambda 2(u'(r_1))^2
	<\frac{\Lambda}{\lambda}G(\xi),
	\]
	but this is in contradiction with \eqref{xi} and, hence, the case (\ref{Nplusa}) is impossible. 
	\\
	Let us now suppose that (\ref{Nplusb}) holds. First of all, let us observe that, by Lemma \ref{le:Lemma21+}, $E_\Lambda$ is decreasing in $[r_1,+\infty)$.
	Therefore there exists $E_{\Lambda,\infty}=\lim_{r\to +\infty}E_\Lambda(r)$. 
Moreover $(u')^2$ admits finite limit as $r\to +\infty$ and so, since $u'<0$, then
\[
\lim_{r\to +\infty}u'(r)=0.
\]
Indeed, otherwise, since $u$ admits a  horizontal asymptote, there would exist $c>0$ and two positive increasing diverging sequences $\{s_n\}_n$ and $\{t_n\}_n$, such that $s_n\in (t_n,t_{n+1})$, $u'(s_n)\le -c$, for $n$ sufficiently large, and $u'(t_n)\to 0$, as $n\to +\infty$.
Therefore we would have
$$E_\Lambda(s_n)\ge \frac{\Lambda}{2}c^2+G(u_\infty),$$ for $n$ sufficiently large, while $$E_\Lambda(t_n)\to G(u_\infty)$$ as $n\to +\infty$, contradicting
the existence of the limit of $E_\Lambda$ at infinity.
\\
Let us prove, now, that for all $n\ge 1$, there exists $s_n>n$ such that $0\le u''(s_n)<1/n$.\\
Suppose by contradiction that there exist $c>0$ and $\bar r>0$ such that either $u''(r)<0$ or $u''(r)\ge c$, for any $r>\bar r$. The first case is impossible, since it would contradict the fact that $u'< 0$ and $u'$ goes to zero at infinity.
Also the second case is impossible, because it would imply that $\lim_{r\to +\infty}u'(r)=+\infty$. 
\\
Hence we can argue that there exists a diverging sequence $\{s_n\}_n$ such that $u'(s_n)<0$ and $0\le u''(s_n)<1/n$. Therefore we have
\[
g(u_\infty)
=\lim_{n\to +\infty}g(u(s_n))
=-\lim_{n\to +\infty}\left(\Lambda u''(s_n)+\frac{(N-1)}{s_n}\lambda u'(s_n)\right)=0.
\] 
Since hypothesis \eqref{g3} implies that $g(s) = 0$ if and only if $|s| \in \{0,\a\}$,  we get a contradiction if $u_\infty \in  (-\alpha, 0)$.\\
Suppose, therefore, that $u_\infty=-\alpha$. Since, by Lemma \ref{le:Lemma21+}, $E_\Lambda$ is decreasing in $[r_1,+\infty)$, we have that, for large $n\in\mathbb{N}$, $E_\Lambda(r_1)>E_\Lambda(s_n)$, namely
	\begin{equation*}
	\frac \Lambda 2(u'(r_1))^2>\frac \Lambda 2 (u'(s_n))^2 +G(u(s_n)).
	\end{equation*} 
	Therefore, passing to the limit on $n$ and using  \eqref{+1} we infer
	\[
	G(\alpha)
	\le \frac \Lambda 2(u'(r_1))^2
	<\frac{\Lambda}{\lambda}G(\xi),
	\]
	but this is in contradiction with \eqref{xi} and, hence, also the case (\ref{Nplusb}) is impossible. 
	\\
	Therefore, we can say that $u'<0$ until $u$ attains a critical point at some $r_2 > r_1$ with $-\a < u(r_2 ) < 0$. 
	\\
	Moreover, by \eqref{eqr}, \eqref{g2} and \eqref{g3}, we have
	\[
	\theta u''(r_2)=-g(u(r_2))>0.
	\]
	Hence, $r_2$ is a local minimizer. 
\end{proof}	
\begin{lemma}\label{le:step3}
$u$ oscillates and  $\|u\|_{L^\infty(\R_+)}<\a$.
\end{lemma}
\begin{proof}
	Let us first show that there exists an increasing sequence $\{r_n\}_n$ such that
	$r_0=0$, each $r_{4k}$ is local maximizer, each $r_{2+4k}$
	is local minimizer, and each $r_{1+2k}$ is zero of $u$, with $k\in \N$. 
	\\
If $u$ remains increasing and  convex, then it necessarily has to increase to a second zero $r_3 > r_2$.
	\\
	Assume, by contradiction that $u<0$ in $[r_2,R)$. Then, arguing as in the proof of Lemma \ref{le:step1}, there cannot exist $\bar r\in [r_2,R)$ such that $u(\bar r)\in (-\a ,0) $, $u'(\bar r)=0$ and $u''(\bar r)\le 0$ and so $u'>0$ therein. Hence $R=+\infty$ and  there exists 
	\[
	\lim_{r\to +\infty}u(r)=u_\infty\in(-\a ,0].
	\]
	Therefore, $u$ satisfies
	\begin{equation*}
	u''=-\frac{N-1}{\theta r}\Lambda u'-\frac{g(u)}{\t}
	\ge -\frac{N-1}{\lambda r}\Lambda u'-\frac{g(u)}{\Lambda}.
	\end{equation*}
	If we set $v(r) = r^\frac{(N-1)\Lambda}{ 2\lambda} u(r)$,  we get 
	\begin{equation*}\label{sigma-}
	v''+\s_- (r)v\ge  0 \qquad\hbox{ where }\
	\s_-(r):= \frac{ g(u(r))}{\Lambda u(r)}-\frac{(N -1)\Lambda[(N-1)\Lambda - 2\l]}{4\lambda^2 r^2}.
	\end{equation*}
	Then we reach a contradiction as in the proof of Lemma \ref{le:step1}, obtaining a second zero $r_3 > r_2$.
	\\
	Concerning the
	behaviour of $u$ on $[r_3 ,+\infty)$, there are again three possibilities:
	\begin{enumerate}[label=(\alph{*}'), ref=\alph{*}']
		\item \label{Nplusap} $u'>0$ until $u$  attains $-u(r_2)$  at some $\bar r>r_3$;
		\item \label{Nplusbp} $u'>0$ on $[r_3 , +\infty)$ and $u$ increases to some value $u_\infty \in (0,-u(r_2)]$;
		\item \label{Npluscp} $u'>0$ until $u$ attains a critical point (which is a maximizer) at some $r_4 > r_3$ with $0 < u(r_4 ) < -u(r_2)$.
	\end{enumerate}
Let us show that the case (\ref{Nplusap}) does not occur.\\
Indeed, if there existed $\bar r > r_3$ such that $u(\bar r) = -u(r_2)$, then we would deduce that
\[
E_\Lambda(\bar r) \ge G(u(\bar r)) = G(u(r_2))= E_\Lambda(r_2),
\]
which is in contradiction with the decreasing monotonicity of $E_\Lambda$ in the  interval $[r_2,\bar r]$ (see Lemma \ref{le:Lemma21+}).
Hence, the case (\ref{Nplusap}) is impossible.\\
	Modifying slightly the arguments of (\ref{Nplusb}) in the proof of Lemma \ref{le:step2}, using again the fact that, by Lemma \ref{le:Lemma21+},  $E_\Lambda$ is  decreasing whenever $u'>0$, we can prove that also the case (\ref{Nplusbp}) cannot occur.\\
	Therefore  $u'>0$ until $u$ attains a critical point, which is actually a local maximizer, at some $r_4 > r_3$ with $0 < u(r_4 ) < -u(r_2)$.\\
	Using Lemma \ref{le:Lemma21+} several times, we now give a better estimate of the value $u(r_4)$.
	\\
	We start observing that $E_\lambda$ is decreasing in $[0,r_1]$ and so $E_\lambda (0)>E_\lambda(r_1)$, which means that
	\begin{equation}\label{a'1+}
	G(\xi)>\frac \lambda 2 (u'(r_1))^2.
	\end{equation} 
	Since $E_\Lambda$ is decreasing in $[r_1,r_3]$ (because $u'<0$ and $g(u)<0$ in $[r_1,r_2)$ and $u'>0$ in $(r_2,r_3)$), we have that $E_\Lambda(r_1)>E_\Lambda(r_3)$, namely
	\begin{equation}\label{a'2+}
	(u'(r_1))^2>(u'(r_3))^2.
	\end{equation} 
	Moreover  $E_\lambda$ is decreasing in $[r_3, r_4]$ and so $E_\lambda (r_3)>E_\lambda(r_4)$, namely
	\begin{equation}\label{a'3+}
	\frac \lambda 2 (u'(r_3))^2>G(u(r_4)).
	\end{equation} 
	Therefore, putting together \eqref{a'1+}, \eqref{a'2+} and \eqref{a'3+}, we have
	\[
	G(\xi)>\frac \lambda 2(u'(r_1))^2
	>\frac \lambda 2(u'(r_3))^2
	>G(u(r_4)),
	\]
	and so $0<u(r_4)<\xi$.
	\\
	We can now repeat the arguments used before proving the existence of a zero
	$r_5 > r_4$, a local minimizer $r_6>r_5$ and so on, such that 
	the sequence of local maxima  $\{u(r_{4k})\}_k$ is decreasing and $\xi=u(r_0)$.
\\
Again by means of Lemma \ref{le:Lemma21+}, we can also say something on the minimizers.\\
Indeed, since $E_\Lambda$ is decreasing in $[r_2,r_3]$, we have that $E_\Lambda (r_2)>E_\Lambda(r_3)$,  and so 
\begin{equation}\label{a'4+}
G(u(r_2))>\frac \Lambda 2 (u'(r_3))^2.
\end{equation} 
Being $E_\lambda$ decreasing in $[r_3,r_5]$,
we have that $E_\lambda(r_3)>E_\lambda(r_5)$, namely
\begin{equation}\label{a'5+}
(u'(r_3))^2>(u'(r_5))^2.
\end{equation} 
Besides,  $E_\Lambda$ is decreasing in $[r_5, r_6]$ and so $E_\Lambda (r_5)>E_\Lambda(r_6)$, namely
\begin{equation}\label{a'6+}
\frac \Lambda 2 (u'(r_5))^2>G(u(r_6)).
\end{equation} 
Therefore, putting together \eqref{a'4+}, \eqref{a'5+} and \eqref{a'6+}, we have
\[
	G(u(r_2))
	>G(u(r_6)).
	\]
Hence, by \eqref{g2} and \eqref{g3}, we conclude that $u(r_2)<u(r_6)$. Repeating these arguments, we have also that  the sequence of local minima  $\{u(r_{2+4k})\}_k$ is increasing.
\\
We conclude these estimates, observing that, since $E_\L$ is decreasing whenever $u'>0$, we deduce that $-u(r_{2+4k})>u(r_{4+4k})$, for any $k\in \N$.
	\\
	Notice, at last, that this reasoning also shows that there are no further zeros or critical points. Moreover we conclude that $\|u\|_{L^\infty(\R_+)}=\max\{\xi,-u(r_2)\}<\alpha$.
\end{proof}

\begin{lemma}\label{lemma0}
	There exists $c>0$ such that $|u'(r)|+|u''(r)|\le c$, for all $r\in \R_+$. \\Moreover, if 
\begin{equation}\label{uto0}
\lim_{r\to +\infty} u(r)=0,
\end{equation}
then we have also that
\begin{equation}\label{u'u''to0}
\lim_{r\to +\infty} |u'(r)|+|u''(r)|=0.
\end{equation}
\end{lemma}

\begin{proof}
	Fix $k\in \N$. 
	Since, by Lemma \ref{le:Lemma21+}, $E_\l$ is decreasing in $[r_{4k},r_{1+4k}]$, then 
	\begin{equation}\label{step1}
	\frac{\l}2 (u'(r))^2
	\le E_\l(r) \le E_\l(r_{4k})\le G(\xi), \qquad \hbox{ for all }r\in[r_{4k},r_{1+4k}],
	\end{equation}
	while, being  $E_\L$ decreasing in $[r_{1+4k},r_{4+4k}]$, by \eqref{step1}, we have 
	\begin{equation}\label{step2}
	\frac{\L}2 (u'(r))^2
	\le E_\L(r) \le E_\L(r_{1+4k})= \frac{\L}2 (u'(r_{1+4k}))^2 \le\frac{\L}{\l}G(\xi),  \ \hbox{ for all }r\in[r_{1+4k},r_{4+4k}].
	\end{equation}
	This implies that $|u'|$ is bounded and so we deduce the boundedness of $|u''|$, since $u$ is a solution of \eqref{eqr}.
\\
Finally, supposing that \eqref{uto0} holds, by \eqref{step1} and \eqref{step2}, one can deduce that
\[
(u'(r))^2
\le \frac 2\l G(u(r_{4k})), \qquad \hbox{ for all }r\in[r_{4k},r_{4+4k}],
\]
from which $\lim_{r\to +\infty}u'(r)=0$ 
and, since $u$ solves \eqref{eqr}, we deduce also \eqref{u'u''to0}.
\end{proof}

\begin{lemma}\label{lemma1}
	The map $E_\L$ is decreasing in $
\displaystyle \I:=\bigcup_{k=0}^{+\infty}[r_{1+4k},r_{2+4k}].
$
\end{lemma}

\begin{proof}
	By Lemma \ref{le:Lemma21+}, we already know that $E_\L$ is decreasing in $[r_{1+4k},r_{2+4k}]$, for any fixed $k\in \N$. The proof is concluded if we show that
	\begin{equation*}
	E_\L(r_{5+4k}) < E_\L(r_{2+4k}), \qquad\hbox{for all }k\in \N.
	\end{equation*}
	For simplicity we prove it just for $k=0$.\\
	Being $E_\L$ decreasing in $[r_{2},r_{3}]$, we have $E_\L(r_3)<  E_\L(r_2)$.
	While, since $E_\l$ is decreasing in $[r_3,r_5]$, we deduce that  $E_\l(r_5)< E_\l(r_3)$, namely
	\[
	(u'(r_5))^2< (u'(r_3))^2,
	\]
	so that also $E_\L(r_5)< E_\L(r_3)$ and we can conclude.
\end{proof}

In order to conclude the proof of  (\ref{n>1+iii}) of Theorem \ref{n>1+}, we need only to prove that $u$ is localized and we will do this in the next proposition. The arguments are inspired by \cite{GZ,MMP} but, being $E_\L$ decreasing only in $\I$ and not on the whole $\R$  and due to the particular nature of the Pucci's extremal operators, we have to use, partially, a different approach. Even if some details could be omitted, referring to \cite{GZ,MMP}, we present the entire proof for the sake of completeness.

\begin{proposition}
$u $ is localized.
\end{proposition}

\begin{proof}
By Lemma \ref{lemma0}, we need only to prove \eqref{uto0}.
\\
Since we know that $0<u(r_{4+4k})<|u(r_{2+4k})|$, for all $k\in \N$, then, in order to prove \eqref{uto0}, we need only to show that
	\begin{equation*}
	\lim_{k\to +\infty} u(r_{2+4k})=0.
	\end{equation*}
	We argue by contradiction and suppose that there exists $\bar{c}>0$ such that
	\begin{equation}\label{assurdo}
	-\alpha <u(r_{2+4k})\le -\bar{c}, \qquad\hbox{for all }k\in \N.
	\end{equation}
	We will show that this implies that $\inf_{\I} E_\L=-\infty $, while we know that $E_\L\ge 0$ in $\I$.\\
Since the proof is quite long and articulated, we divide it into intermediate steps.
\\ 
	{\it Claim 1: there exists $c>0$ such that $u'(r_{1+4k})\le - c$, for all $k\in \N$.}
	\\
	Indeed, by Lemma \ref{lemma1} and \eqref{assurdo}, for any $k\ge 0$ we have that
	\[
	\frac{\L}{2}(u'(r_{1+4k}))^2
	\ge \inf_{\I}E_\L
	=\lim_{k\to +\infty} E_\L(r_{2+4k})
	\ge G(\bar{c}).
	\]
	{\it Claim 2:  there exist $\d_1 >0$ and $c>0$ such that $u'(r)\le -c$, for all $k\in \N$ and $r\in [r_{1+4k}-\d_1 ,r_{1+4k}+\d_1]$.}
	\\
	We prove the claim in $[r_{1+4k},r_{1+4k}+\d_1]$ and with similar arguments we can extend the result in the whole interval.
	\\
	Suppose by contradiction that our claim does not hold, then there exists a diverging sequence $\{k_n\}_n$ such that, for any $n\ge 1$, there exists $s_n\in [r_{1+4k_n},r_{1+4k_n}+\frac 1n]$ with $\lim_{n\to +\infty} u'(s_n)=0$. 
	But this is in contradiction with Claim 1, due to boundedness of $|u''|$ stated in Lemma \ref{lemma0}.\\
	{\it Claim 3: there exists $c>0$ such that $u(r_{4k})\ge c$, for all $k\in \N$.}
	\\
	By Claim 2, for any $k\ge 0$ we have
	\[
	u(r_{4k})\ge u(r_{1+4k}-\d_1)
	= u(r_{1+4k}-\d_1) -u(r_{1+4k})
	=-\int_{r_{1+4k}-\d_1}^{r_{1+4k}}u'(r)\, dr
	\ge c\d_1 .
	\] 
	{\it Claim 4: there exists $c>0$ such that $u'(r_{3+4k})\ge c$, for all $k\in \N$.}
	\\
	Fix $k\ge 0$. Since $u$ is increasing in $[r_{3+4k},r_{4+4k}]$, then, by Lemma \ref{le:Lemma21+}, $E_\L$ is decreasing in $[r_{3+4k},r_{4+4k}]$, then 
	\[
	\frac{\L}{2}(u'(r_{3+4k}))^2> G(u(r_{4+4k})),
	\]
	and we conclude by Claim 3.\\
	{\it Claim 5:  there exist $\d_2 >0$ and $c>0$ such that $u'(r)\ge c$, for all $k\in \N$ and $r\in [r_{3+4k}-\d_2 ,r_{3+4k}+\d_2]$.}\\
	The proof is similar to that of Claim 2, using Claim 4.
	\\
	{\it Claim 6:  there exist $c_1,c_2>0$ such that 
		\[
		c_1\le |u(r)|\le c_2<\alpha, \quad \hbox{ for all $k\in \N$ and }r\in \R_+\setminus \bigcup_{k=0}^{+\infty} [r_{1+2k}-\d ,r_{1+2k}+\d]\]
where $\d :=\min\{\d _1,\d_2\}$.}
	\\
	We need to prove only the control from below of $|u|$ which is a trivial consequence of Claim 2 and Claim 5.\\
	{\it Claim 7:  there exists $R>0$ such that $r_{5+4k}-r_{1+4k}\le R$, for all $k\in \N$.}
	\\
	Suppose by contradiction that there exists a diverging sequence $\{k_n\}_n$ such that
	\[
	r_{5+4k_n}-r_{1+4k_n}\xrightarrow[n\to +\infty]{}+\infty,
	\]
	hence there exists $i=1,\ldots,4$ such that, up to a subsequence, still denoted by $k_n$,
	\[
	r_{i+1+4k_n}-r_{i+4k_n}\xrightarrow[n\to +\infty]{}+\infty.
	\]
	Being  the arguments similar, we can suppose that 
	\[
	r_{5+4k_n}-r_{4+4k_n}\xrightarrow[n\to +\infty]{}+\infty.
	\]
	Since $u$ solves \eqref{eqr}, by Lemma \ref{lemma0} and Claim 6, we infer that there exists $c>0$ such that,  for $n$ large enough,
	\[
	u''(r)\le -c,\qquad
	\text{ for }r\in [r_{4+4k_n},r_{5+4k_n}-\d].
	\]
	Hence we have that
	\[
	u'(r_{5+4k_n}-\d)
	=\int_{r_{4+4k_n}}^{r_{5+4k_n}-\d}u''(r)\, dr
	\le -c (r_{5+4k_n}-\d-r_{4+4k_n})\xrightarrow[n\to +\infty]{}-\infty
	\]
	which is contradiction with Lemma \ref{lemma0}.
	\\
	We now reach a contradiction, concluding the proof, if we prove the final step.
	\\
	{\it Claim 8:  $\inf_{\I} E_\L=-\infty.$}
	\\
	Since $u$, in $\I$, solves 
	\[
	\L u''(r)+\l \frac{N-1}{r}u'(r)+g(u(r))=0,
	\] 
	then, for any $r\in \I$, we have
	\[
	E'_\L(r)=-\l \frac{N-1}{r}(u'(r))^2.
	\]
	Hence, by Claim 2,
	\begin{equation}\label{dr1}
	\begin{split}
	E_\L(r_1+\d)
	&=E_\L(r_1)
	+\int_{r_1}^{r_1+\d} E'_\L(r)\, dr
	\\
	&=E_\L(r_1)
	-\int_{r_1}^{r_1+\d} \l \frac{N-1}{r}(u'(r))^2\, dr
	\\
	&\le E_\L(r_1)
	-c\log\left(1+\frac{\d}{r_1}\right)
	\\
	&\le E_\L(r_1)-c\frac{\d}{r_1}, 
	\end{split}
	\end{equation}
	where we have used the fact that there exists a fixed positive constant $c$ such that for any $s\in (0,\d /r_1]$, we have $\log(1+s)\ge cs$.
	\\
	Analogously, by \eqref{dr1} and Lemma \ref{lemma1}, we deduce that
	\[
	E_\L(r_5+\d)
	\le E_\L(r_5)
	-c\log\left(1+\frac{\d}{r_5}\right)
	\le E_\L(r_1)-c\d\left(\frac{1}{r_1}+\frac{1}{r_5}\right).
	\]
	Repeating this argument and since by Claim 7 we have that $r_{1+4k}\le r_1+kR$, for any $k\in \N$, we infer that
	\[
	E_\L(r_{1+4k}+\d)
	\le E_\L(r_1)-c\d\sum_{j=0}^{k}\frac{1}{r_1+jR},
	\]
	which implies that $\inf_{\I} E_\L=-\infty$, as desired.
\end{proof}

\begin{remark}
While in one-dimensional case we have an entire description of the situation for any $0<\xi<\a$, in the multi-dimensional one the partial results on the monotonicity of $E_\l$ and $E_\L$ seem to be not enough to understand completely what happens whenever $\L G(\xi)>\l G(\a)$.
\end{remark}

\subsection{Proof of (\ref{n>1+iv}) of Theorem \ref{n>1+}}

\
\\
		Arguing again as in \eqref{behav0} we infer that 
		$N\L u''(0)=-g(\xi)>0$
		and $u$ is increasing and convex near zero and, also in this case, we divide the proof into intermediate steps.\\
		{\it Step 1: $u'>0$ until $u$ attains a first zero in a certain $r_1>0$}.\\
		The first zero $r_1$ can be obtained arguing as in the proof of Lemma \ref{le:step1}.
		\\
		{\it Step 2: $u'>0$ until $u$ attains a critical point (which is also a local maximizer) at some $r_2 > r_1$ with $0 < u(r_2 ) < -\xi$}.
		\\
		Concerning the
		behaviour of $u$ on $[r_1 ,+\infty)$, there are the following three possibilities:
		\begin{enumerate}[label=(\alph{*}), ref=\alph{*}]
			\item \label{Nminusa}$u'>0$ until $u$ attains $-\xi$  at some $\bar r>r_1$;
			\item \label{Nminusb}$u'>0$ on $[r_1 , +\infty)$ and $u$ increases  to some value $u_\infty \in  (0,-\xi]$;
			\item \label{Nminusc}$u'>0$ until $u$  attains a critical point at some $r_2 > r_1$ with $0 < u(r_2 ) < -\xi$.
		\end{enumerate}
		Let us show that the cases (\ref{Nminusa}) and (\ref{Nminusb}) do not occur.\\
		Indeed, if there existed $\bar r > r_1$ such that $u(\bar r) = -\xi$, then we would deduce that
		\[
		E_\lambda(\bar r) \ge G(u(\bar r)) = G(\xi)=G(u(0)) = E_\lambda(0),
		\]
		which is in contradiction with the decreasing monotonicity of $E_\lambda$ in the  interval $[0,\bar r]$ (see Lemma \ref{le:Lemma21+}).
		Hence, the case (\ref{Nminusa}) is impossible. 
		\\
		Let us now suppose that (\ref{Nminusb}) holds.\\
		First of all, let us observe that, by Lemma \ref{le:Lemma21+} $E_\lambda$ is decreasing in $[r_1,+\infty)$.
		Therefore there exists $E_{\lambda,\infty}:=\lim_{r\to +\infty}E_\lambda(r)$.
		\\
		Arguing as in the case (\ref{n>1+iii}), we can prove that $\lim_{r\to +\infty}u'(r)=0$ and
		that there exists a diverging sequence $\{s_n\}_n$ such that $u'(s_n)>0$ and $-1/n< u''(s_n)\le0$. Therefore we have
		\[
		g(u_\infty)
		=\lim_{n\to +\infty}g(u(s_n))
		=-\lim_{n\to +\infty}\left(\lambda u''(s_n)+\frac{(N-1)}{s_n}\Lambda u'(s_n)\right)=0
		\] 
		and so we get a contradiction being $u_\infty \in  (0,-\xi]$. 
		\\
		Therefore, we can say that $u'>0$ until $u$  attains a critical point at some $r_2 > r_1$ with $0 < u(r_2 ) <- \xi$. 
		Finally, by \eqref{eqr}, \eqref{g2} and \eqref{g3}, we have
		\[
		\theta u''(r_2)=-g(u(r_2))<0.
		\]
		Hence, $r_2$ is a local maximizer. 
		\\
		{\it Step 3: $u$ oscillates and  $\|u\|_{L^\infty(\R_+)}=-\xi$}.
		\\
		Let us first show that there exists an increasing sequence $\{r_n\}_n$ such that $r_0=0$, each $r_{4k}$ is local minimizer, each $r_{2+4k}$ is local maximizer, and each $r_{1+2k}$ is zero of $u$,  with $k\in\N$. 
		\\
		Arguing as in the proof of Lemma \ref{le:step3},  we can infer that $u'<0$  in a right neighbourhood of $r_2$, until $u$ remains positive, and we get a second zero $r_3 > r_2$.\\
		Concerning the behaviour of $u$ on $[r_3 ,+\infty)$, there are again three possibilities:
		\begin{enumerate}[label=(\alph{*}'), ref=\alph{*}']
			\item \label{Nminusap} $u'<0$ until $u$  attains $\xi$  at some $\bar r>r_3$;
			\item \label{Nminusbp} $u'<0$ on $[r_3 , +\infty)$ and $u$ decreases to some value $u_\infty \in [\xi,0)$;
			\item \label{Nminuscp} $u'<0$ until $u$ attains a critical point (which is a local minimizer) at some $r_4 > r_3$ with $\xi < u(r_4 ) < 0$.
		\end{enumerate}
		Let us prove that the case (\ref{Nminusap}) cannot occur.\\
		The arguments cannot be the same of Step 2, because the maps $E_\lambda$ and $E_\Lambda$ are not always decreasing whenever $u'<0$.
		\\
		We start observing that, by Lemma \ref{le:Lemma21+}, $E_\Lambda$ is decreasing in $[0,r_1]$ and so $E_\Lambda (0)>E_\Lambda(r_1)$, which implies that
		\begin{equation}\label{a'1}
		G(\xi)>\frac \Lambda 2 (u'(r_1))^2.
		\end{equation}
		Since, again by Lemma \ref{le:Lemma21+}, $E_\lambda$ is decreasing in $[r_1,r_3]$, 
		we have that $E_\lambda(r_1)>E_\lambda(r_3)$, namely
		\begin{equation}\label{a'2}
		(u'(r_1))^2>(u'(r_3))^2.
		\end{equation}  
		If (\ref{Nminusap}) held, then, being $u'<0$ and $g(u)<0$ in $(r_3,\bar r]$, by Lemma \ref{le:Lemma21+}, we would have that $E_\Lambda$ is decreasing in $[r_3, \bar r]$ and so $E_\Lambda (r_3)>E_\Lambda(\bar r)$, namely
		\begin{equation}\label{a'3}
		\frac \Lambda 2 (u'(r_3))^2>
		\frac \Lambda 2 (u'(\bar r))^2 +G(\xi).
		\end{equation} 
		Therefore, putting together \eqref{a'1}, \eqref{a'2} and \eqref{a'3}, we have
		\[
		G(\xi)>\frac \Lambda 2(u'(r_1))^2
		>\frac \Lambda 2(u'(r_3))^2
		>\frac \Lambda 2 (u'(\bar r))^2 +G(\xi)
\ge G(\xi),
		\]
		and we reach a contradiction.
		\\
		Finally we can prove that the case (\ref{Nminusbp}) does not occur with similar arguments to those of  (\ref{Nminusb}) of  Step 2 using the fact that $E_\Lambda$ would be decreasing in $[r_3, +\infty)$.
		\\
		Similarly to the case (\ref{n>1+iii}), we can now repeat the arguments proving the existence of
		a zero $r_5>r_4$, a local maximizer $r_6>r_5$ and so on, such that the sequence of the local minima $\{u(r_{4k})\}_k $ is increasing, the sequence of local maxima $\{u(r_{2+4k})\}_k $ is decreasing. Moreover we have also that $u(r_{2+4k})<-u(r_{4k})$, for any $k\in \N$.
		Notice, at last, that this reasoning also shows that there are no further zeros or critical points and that $\|u\|_{L^\infty(\R_+)}=-\xi$.
		\\
		{\it Step 4: $u$ is localized}.
\\
 We can prove that $u$ is localized repeating the arguments of the case \eqref{n>1+iii} defining
\begin{equation*}
\I':=\bigcup_{k=0}^{+\infty}[r_{3+4k},r_{4+4k}],
\end{equation*}
and observing that $E_\L$ is decreasing in $\I'$.

\subsection*{Acknowledgment}

The authors wish to express their more sincere gratitude to the anonymous referee: his/her acute comments and suggestions have been crucial to improve the quality and the clarity of paper.

\end{document}